\documentclass[pdflatex,sn-mathphys]{sn-jnl}

\usepackage{multicol}
\newtheorem{lemma}{Lemma}

\newtheorem{conjecture}{Conjecture}
\newtheorem{openproblem}{Open Problem}
\jyear{2021}%

\theoremstyle{thmstyleone}%
\newtheorem{theorem}{Theorem}
\theoremstyle{thmstyletwo}%
\newtheorem{example}{Example}%

\theoremstyle{thmstylethree}%
\newtheorem{definition}{Definition}%

\begin{document}

\title[Multipartite Ramsey numbers of complete bipartite graphs]{Multipartite Ramsey numbers of complete bipartite graphs arising
from algebraic combinatorial structures}


\author[1]{\fnm{I Wayan Palton} \sur{Anuwiksa}}\email{paltonanuwiksa@students.itb.ac.id}

\author*[2,3]{\fnm{Rinovia} \sur{Simanjuntak}}\email{rino@itb.ac.id}

\author[2,3]{\fnm{Edy Tri} \sur{Baskoro}}\email{ebaskoro@itb.ac.id}

\affil[1]{\orgdiv{Doctoral Program of Mathematics, Faculty of Mathematics and Natural Sciences}, \orgname{Institut Teknologi Bandung}, \orgaddress{\country{Indonesia}}}

\affil[2]{\orgdiv{Combinatorial Mathematics Research Group}, \orgname{Institut Teknologi Bandung}, \orgaddress{\street{Jl Ganesa 10}, \city{Bandung}, \postcode{40132}, \country{Indonesia}}}

\affil[3]{\orgname{Center for Research Collaboration on Graph Theory and Combinatorics}, \orgaddress{\city{Bandung}, \country{Indonesia}}}


\abstract{In 2019, Perondi and Carmelo determined the set multipartite Ramsey number of particular complete bipartite graphs by establishing a relationship between the set multipartite Ramsey number, Hadamard matrices, and strongly regular graphs, which is a breakthrough in Ramsey theory. However, since Hadamard matrices of order not divisible by 4 do not exist, many open problems have arisen.\\ 

In this paper, we generalize Perondi and Carmelo's results by introducing the $[\alpha]$-Hadamard matrix that we conjecture exists for arbitrary order. 
Finally, we determine set and size multipartite Ramsey numbers for particular complete bipartite graphs.}




\maketitle

\section{Introduction}\label{Intro}

In 2004, Burger and Van Vuuren \cite{P1,P2} introduced the notion of set multipartite Ramsey number and size multipartite Ramsey number as variations of the classical Ramsey number. The extension to many colors is established in \cite{P3}, and the extension in a more general setting is presented in \cite{P7,P8}. For $c \geq 2,s \geq 1$, we denote by $K_{c \times s}$ the complete multipartite graph with $c$ partite sets, each of which contains $s$ vertices. 

\begin{definition}\label{definisi1} \cite{P7,P8}
Let $s,k$ be positive integers with $k\geq 2$ and $G_1,G_2,\ldots,G_k$ be simple graphs. The  \emph{set multipartite Ramsey number}, denoted by $M_s(G_1,G_2,\ldots,G_k)$, is the smallest positive integer $c$ such that any $k$-coloring of the edges of $K_{c\times s}$ contains a monochromatic copy of $G_i$ in color $i$ for some $i\in\{1,2,\ldots,k\}$.

The \emph{size multipartite Ramsey number}, denoted by $m_c(G_1,G_2,\ldots,G_k)$, is the smallest positive integer $s$ such that any $k$-coloring of the edges of $K_{c\times s}$ contains a monochromatic copy of $G_i$ in color $i$ for some $i\in\{1,2,\ldots,k\}$.

In the case of $G_1=G_2=\ldots= G_k=G$, the two aforementioned Ramsey numbers are abbreviated to $M_s(G;k)$ and $m_c(G;k)$, respectively.
\end{definition} 

Surprisingly, there is a strong relation among set and size multipartite Ramsey numbers, strongly regular graphs, and Hadamard matrices. In 2019, Perondi and Carmelo \cite{P7}  determined the set multipartite Ramsey number of bipartite graphs by establishing the relationship of set multipartite Ramsey number, Hadamard matrices, and strongly regular graphs. Their results are presented in Theorems \ref{A2} and \ref{A1}.

\begin{theorem} \cite{P7} \label{A2}
Suppose that there is a strongly regular graph with parameters $(4n-3, 2n-2, n-2, n-1)$ and there is a symmetric
Hadamard matrix of order $\zeta$. Then
 $M_\zeta(K_{2,\zeta(n-1)+1};2) \geq  4n-2$.
\end{theorem}

\begin{theorem} \cite{P7} \label{A1}
Suppose that there is a  strongly regular graph with parameters $(4n-3, 2n-2, n-2, n-1)$ and there is a  symmetric
Hadamard matrix of order $\zeta$ with $\zeta \geq 4n$. Then
 $M_\zeta(K_{2,\zeta(n-1)+1};2) =  4n-2$.
\end{theorem}

It is well-known that Hadamard matrices of order $\zeta\not\equiv 0\pmod 4$ do not exist, which means the conclusions of Theorems \ref{A2} and \ref{A1} only apply for $\zeta\equiv 0\pmod 4$ at best. This fact motivates us to introduce a generalization of the Hadamard matrix that we call the $[\alpha]$-Hadamard matrix and the $\alpha$-Hadamard matrix, defined in the following. 

\begin{definition}\label{defmatrix}
Let $\zeta\geq 1$ and $\alpha \in\{0,1,\ldots,\zeta\}$. $H$ is an \emph{$[\alpha]$-Hadamard matrix} (resp. \emph{$\alpha$-Hadamard matrix}) of order $\zeta$ if and only if $H$ is a square matrix of order $\zeta$ with entries $1$ or $-1$ and $\alpha$ is the upper bound (resp. maximum) of  
$\{\mid HH^t(i,j)\mid\,\mid\,1\leq i \neq j\leq \zeta\}$.
\end{definition}

Utilizing the $[\alpha]$-Hadamard matrices, we present our main results in the following two theorems. 


\begin{theorem}\label{Q12} 
Let $G$ be a strongly regular graph with parameters $(n,k,\lambda, \mu)$, $H$ be a symmetric $[\alpha]$-Hadamard matrix of order $\zeta\geq 2$, and $\theta=\max\{k/2,\lambda,\mu,(n-k-1)/2,$ $n-2-2k+\mu,n-2k+\lambda\}$. Then 
$$n+1\leq M_\zeta(K_{2,\theta (\zeta+\alpha)+1};2) \text{ and } \zeta+1\leq m_{n}(K_{2,\theta (\zeta+\alpha)+1};2).$$
Furthermore,
\begin{enumerate}
    \item[\rm \textit{$i).$}] If $\zeta$ or $\left\lceil \frac{4\theta\alpha+1}{\zeta}  \right\rceil-1$ are even, then
$n+1\leq M_\zeta(K_{2,\theta (\zeta+\alpha)+1};2)\leq 4\theta +2 +\left\lceil \frac{4\theta\alpha+1}{\zeta}  \right\rceil$.
\item[\rm \textit{$ii).$}] If $n-1$ or $\left\lceil \frac{4n \theta(\zeta+\alpha)}{(n -1)^2}+\frac{2}{n-1}  \right\rceil$ are even, then $\zeta+1\leq m_{n}(K_{2,\theta (\zeta+\alpha)+1};2)\leq \left\lceil \frac{4n \theta(\zeta+\alpha)}{(n-1)^2}+\frac{2}{n-1}  \right\rceil$.
\end{enumerate}
\end{theorem} 

\begin{theorem} \label{A11}
Suppose that there exist a strongly regular graph with parameters $(4n-3, 2n-2, n-2, n-1)$ and a symmetric
$[\alpha]$-Hadamard matrix of even order $\zeta$ with $\zeta > (\sqrt{2}+1)(2n-1)(4\alpha n-4\alpha+1)$. Then
$$M_\zeta(K_{2,(\zeta+\alpha)(n-1)+1};2) = 4n-2.$$
\end{theorem}

Since we conjecture that a symmetric $[\alpha]$-Hadamard matrix of order $\zeta$ exists for every $\zeta$, the conclusions in our results apply for all $\zeta$ at best, that can be viewed as a generalization of Theorems \ref{A2} and \ref{A1}. 

The rest of the paper is organized as follows. In Section \ref{sec:Hadamard}, we provide constructions of an $[\alpha]$-Hadamard matrix and conjecture the existence of $[\alpha]$-Hadamard matrices of all orders. The proof and examples of the application of Theorem \ref{Q12} are presented in Section \ref{sec:Ramsey}, and those of Theorem \ref{A11} are given in Section \ref{sec:exact}. Finally, we conclude by proposing some open problems in Section \ref{OPD}.

\section{$[\alpha]$-Hadamard and $\alpha$-Hadamard Matrices} \label{sec:Hadamard}


It is easy to see that a $0$-Hadamard matrix is a Hadamard matrix. Moreover, if $H$ is a matrix of order $m$ whose entries are $1$ or $-1$, then it is also clear that there exists an $\alpha$ such that $H$ is an $\alpha$-Hadamard matrix. Another important observation is that if $H$ is an $\alpha$-Hadamard matrix, it is also an $[\alpha]$-Hadamard matrix. 

An obvious observation for an $[\alpha]$-Hadamard (resp. an $\alpha$-Hadamard) matrix $H$ is that if 
\begin{enumerate}
    \item the rows of H are permuted, or 
    \item the columns of H are permuted, or 
    \item any row of $H$ is multiplied  by -1, or
    \item any column of $H$ is multiplied by -1, then
\end{enumerate}
the resulting matrix is still $[\alpha]$-Hadamard (resp. an $\alpha$-Hadamard).

In the following theorem, we provide a construction of an $[\alpha]$-Hadamard (resp. a symmetric $\alpha$-Hadamard matrices) from a Hadamard matrix (resp. a symmetric Hadamard matrix).

\begin{theorem}\label{H1}
Let $\alpha$ and $\zeta$ be non-negative integers. If $ \alpha\leq \zeta/2$ and there exists a Hadamard matrix (resp. a symmetric Hadamard matrix) of order $\zeta$, then there exists an $[\alpha]$-Hadamard matrix (resp. a symmetric $\alpha$-Hadamard matrix) of order $\zeta-\alpha$.
\end{theorem}

\begin{proof}
Let $H$ be a Hadamard matrix of order $\zeta$. For $1\leq i\leq \zeta$, let $r_i$ be the $i^{th}$ row of $H$. Then $\left<r_i,r_j\right>=0,\,\forall i\neq j\in\{1,2,\ldots ,\zeta\}$. Let $A_1, A_2\subset \{0,1,\ldots,\zeta\}$, where $\lvert A_1\rvert=\lvert A_2\rvert=\alpha \leq \zeta/2$. To prove the first part of the theorem, we construct a matrix $H'$ by deleting the $i^{th}$ row and the $j^{th}$ column of $H$, for any $i\in A_1$ and $j \in A_2$. If $r_i'$ is the $i^{th}$ row of $H'$, then $\lvert \left<r_i',r_j'\right>\rvert \leq \alpha$, for all $i\neq j\in\{1,2,\ldots,\zeta-\alpha\}$. We conclude that $H'$ is an $[\alpha]$-Hadamard matrix of order $\zeta-\alpha$.

For the second part of the theorem, assume that $H$ is a symmetric Hadamard matrix. Since multiplying any row and column of $H$ by $-1$ does not change the property of $\lvert \left<r_i,r_j\right>\rvert=0$, for all $i\neq j\in\{1,2,\ldots,\zeta\}$, then we may assume that the first row of $H$ is $[1\, 1\, \ldots \, 1]$. Since $\lvert \left<r_1,r_2\right>\rvert=0$,
we conclude that the number of entries $1$ and $-1$ in the second row is equal, which is $\zeta/2$. Now, let $C\subseteq\{j\,\mid \,H(2,j)=-1\}$ where $\lvert C\rvert =\alpha$. Construct a matrix $H''$ by deleting the $i^{th}$ row and the $i^{th}$ column of $H$, for every $i\in C$. Let $r_i''$ be the $i^{th}$ row of $H''$, then $\lvert \left<r_i'',r_j''\right>\rvert \leq \alpha$ for all $i\neq j\in\{1,2,\ldots,\zeta-\alpha\}$, and $\lvert\left<r_1'',r_2''\right>\rvert= \alpha$. Therefore, $H''$ is a symmetric $\alpha$-Hadamard matrix of order $\zeta-\alpha$. 
\end{proof}

\begin{figure}[h]
\begin{multicols}{2}
$$H_1=\left[\begin{array}{cccccccc}
     1&1&1&1&1&1&1&1  \\
     1&-1&1&-1&1&-1&1&-1\\
     1&1&-1&-1&1&1&-1&-1\\
     1&-1&-1&1&1&-1&-1&1\\
     1&1&1&1&-1&-1&-1&-1  \\
     1&-1&1&-1&-1&1&-1&1\\
     1&1&-1&-1&-1&-1&1&1\\
     1&-1&-1&1&-1&1&1&-1
\end{array}\right]$$

$$H_2=\left[\begin{array}{cccccc}
     1&1&1&1&1&1 \\
     1&-1&1&-1&1&1\\
     1&1&-1&-1&1&-1\\
     1&-1&-1&1&1&-1\\
     1&1&1&1&-1&-1  \\
     1&1&-1&-1&-1&1
\end{array}\right]$$
\end{multicols}
\caption{A symmetric Hadamard matrix of order 8, $H_1$, and a symmetric $2$-Hadamard matrix of order 6, $H_2$, constructed by deleting the $6^{th},8^{th}$ rows and the $6^{th},8^{th}$ columns of $H_1$.}
\end{figure}



A well-known conjecture stated the existence of a (symmetric) Hadamard matrix of order $\zeta\equiv 0\pmod 4$. If this conjecture is true, then by Theorem $\ref{H1}$, we obtain a $[1]$-Hadamard matrix and a symmetric $1$-Hadamard matrix of order $p\equiv 3\pmod 4$, a $[2]$-Hadamard matrix and a symmetric $2$-Hadamard matrix of order $q\equiv 2\pmod 4$, and a $[3]$-Hadamard matrix and a symmetric $3$-Hadamard matrix of order $r\equiv 1\pmod 4$. Since a $[1]$-Hadamard matrix and a $[2]$-Hadamard matrix are also a $[3]$-Hadamard matrix, and an $\alpha$-Hadamard matrix is also an $[\alpha]$-Hadamard matrix, we propose the following conjecture.

\begin{conjecture}\label{C11}
For any $\alpha\geq 3$, there exist (symmetric) $[\alpha]$-Hadamard matrices of all orders.
\end{conjecture}

Notice that Conjecture \ref{C11} might be true without the existence of Hadamard matrices. An autonomous construction of $[\alpha]$-Hadamard matrix can be introduced by considering that $\lvert HH^t(i,j)\rvert$ might be distinct for some $i$ and $j$. In the following two theorems, we provide constructions of $1$-Hadamard and $4$-Hadamard matrices of particular order inspired by Paley's construction.

\begin{theorem}\label{1hadamard}
Let $p$ be a prime number and $\alpha>0$. If $p^{\alpha}\equiv 3\pmod 4$, then there exists a $1$-Hadamard matrix of order $p^{\alpha}$.
\end{theorem}
\begin{proof}
Label the elements of $GF(p^{\alpha})$ as $a_0,a_1,a_2,\ldots$ in a particular order. Let $Q=(q_{i,j})$ be a matrix of order $\zeta=p^{\alpha}$ whose entries are given by $q_{ij}=\chi(a_i-a_j)$ where $\chi$ is the quadratic character on $GF(p^{\alpha})$. That is,
$$\chi(b)=\begin{cases}
 0,&\text{if }b=0\\
 +1,&\text{if }b \text{ is a non-zero perfect square in }GF(p^{\alpha})\\
 -1,&\text{if }b \text{ is not a perfect square in }GF(p^{\alpha})
\end{cases}$$
Form the $(\zeta+1)\times (\zeta+1)$ matrix $C=\left[
\begin{array}{cc}
0&\textbf{1}^t\\
-\textbf{1}&Q
\end{array}
\right]$ and let $H=I_\zeta+Q$. This results in $C=-C^t$ (anti-symmetric). Additionally, $C$ has all $0$'s along the diagonal and $\pm1$ elsewhere, and so $CC^t=\zeta I_{\zeta+1}$ (conference matrix). Therefore, $C$ an anti-symmetric conference matrix. 

Note that $q_{ji}=\chi((-1)(a_i-a_j))=\chi(-1)\chi(a_i-a_j)=(-1)\chi(a_i-a_j)=-q_{i,j}$, and so $Q^t=-Q$, $H^t\neq H$, and $H^t\neq -H$. 

We shall show that $H$ is a $1$-Hadamard matrix. First, note that
$$\zeta I_{\zeta+1}=CC^t=\left[
\begin{array}{cc}
0&\textbf{1}^t\\
-\textbf{1}&Q
\end{array}
\right]
\left[
\begin{array}{cc}
0&\textbf{1}^t\\
-\textbf{1}&Q
\end{array}
\right]^t
=\left[
\begin{array}{cc}
\zeta-1&(Q\textbf{1})^t\\
Q\textbf{1}&J_{\zeta}+QQ^t
\end{array}
\right],$$
where $J_{\zeta}$ is a matrix of order $\zeta$ with entry of all $1$. Thus,  $J_{\zeta}+QQ^t=\zeta I_\zeta$, and this results in
$$HH^t=(I_\zeta+Q)(I_\zeta+Q^t)=I_\zeta+QQ^t=(\zeta+1)I_\zeta-J_\zeta.$$
Therefore, $HH^t$ has all $\zeta$'s along the diagonal and $-1$ elsewhere. 
\end{proof}

\begin{theorem}\label{4hadamard}
Let $p$ be a prime number and $\alpha>0$. If $p^{\alpha}\equiv 1\pmod 4$, then there exists a $4$-Hadamard matrix of order $2p^{\alpha}$.
\end{theorem}
\begin{proof} 
Again, let $Q=(q_{ij})$ be a matrix of order $p^{\alpha}$ whose entries are given by $q_{ij}=\chi(a_i-a_j)$. Let $n=p^{\alpha}+1$ and $$H=Q\otimes \left(
\begin{array}{cc}
1&1\\
1&-1
\end{array}
\right)+I_{n-1}\otimes \left(
\begin{array}{cc}
1&-1\\
-1&-1
\end{array}
\right),$$ 
where $\otimes$ denotes the Kronecker product.

Note that $H^t\neq H$ and $H^t\neq -H$. However, $HH^t=2nI_{2n-2}+(J_{n-1}+Q)\otimes \left(
\begin{array}{cc}
-2&-4\\
4&-2
\end{array}
\right)$. This completes the proof.
\end{proof}

The proofs of Theorems \ref{A2} and \ref{A1} depend on  the cardinalities of the partition sets of a Hadamard matrix's column indexes (see Lemma 11 in \cite{P7}). In the following lemma, we generalize that result by counting cardinalities of the partition sets of an $[\alpha]$-Hadamard matrix's column indexes. 

\begin{lemma} \label{lemmahadamard}
Let $H=(h_{i,j})$ be an $[\alpha]$-Hadamard matrix of order $\zeta\geq 2$. For every two distinct integers $i,j$ where $1\leq i,j\leq \zeta$, define 
       \begin{center}
       $I_1=\{  k\in \{1,2,\ldots,\zeta\}   \mid h_{i,k}=1 \text{ and }  h_{j,k}=1   \},$\\
       $I_2=\{  k\in \{1,2,\ldots,\zeta\}   \mid h_{i,k}=-1 \text{ and } h_{j,k}=-1  \},$\\     
       $I_3=\{  k\in \{1,2,\ldots,\zeta\}   \mid h_{i,k}=1  \text{ and } h_{j,k}=-1   \},$\\
       $I_4=\{  k\in \{1,2,\ldots,\zeta\}   \mid h_{i,k}=-1 \text{ and } h_{j,k}=1   \}.$
       \end{center}
Then $\lvert I_1\rvert + \lvert I_2\rvert + \lvert I_3\rvert + \lvert I_4\rvert =\zeta$, $\lvert I_1\rvert + \lvert I_2\rvert \leq \frac{\zeta+\alpha}{2}$, and $\lvert I_3\rvert + \lvert I_4\rvert \leq \frac{\zeta+\alpha}{2}$. Furthermore, if $H$ is an $\alpha$-Hadamard matrix, then $\zeta$ and $\alpha$ have the same parity.
\end{lemma}
\begin{proof} Since $I_1, I_2, I_3, I_4$ partitions the set of $H$'s column indexes, then $\lvert I_1\rvert + \lvert I_2\rvert + \lvert I_3\rvert + \lvert I_4\rvert =\zeta$.
Let $i\neq j$ where $1\leq i,j\leq \zeta$. Since $\lvert HH^t(i,j)\rvert \leq \alpha$, then $\lvert \sum_{k=1}^{\zeta}h_{i,k}h_{j,k}\rvert \leq \alpha$, and so
    $\lvert \lvert I_1\rvert + \lvert I_2\rvert - \lvert I_3\rvert - \lvert I_4\rvert \rvert \leq \alpha.$
Since $\lvert I_1\rvert + \lvert I_2\rvert + \lvert I_3\rvert + \lvert I_4\rvert =\zeta$, then
     $\lvert I_1\rvert + \lvert I_2\rvert \leq \frac{\zeta+\alpha}{2}$ and 
     $\lvert I_3\rvert + \lvert I_4\rvert \leq \frac{\zeta+\alpha}{2}$.
    
Furthermore, if $H$ is an $\alpha$-Hadamard matrix, then there exist $i',j'$, $1\leq i' \neq j'\leq \zeta$, such that $\lvert HH^t(i',j')\rvert = \alpha$, which leads to $\lvert I_3\rvert + \lvert I_4\rvert = \frac{\zeta-\alpha}{2}$ and $\lvert I_1\rvert + \lvert I_2\rvert = \frac{\zeta+\alpha}{2}$. Thus $\zeta$ and $\alpha$ have the same parity. 
\end{proof}


\section{Proof of Theorem \ref{Q12}}
\label{sec:Ramsey}


Given a vertex $v$ of a graph $G = (V, E)$, we denote by $N(v)$ the set of neighbors of $v$. Recall that a graph $G$ is strongly
regular with parameters $(n, k, \lambda,\mu)$ when
\begin{itemize}
    \item $G$ has $n$ vertices;
    \item $G$ is $k$-regular;
    \item if $vw$ is an edge of $G$, then $\mid N(v) \cap N(w)\mid  = \lambda$;
    \item if $vw$ is not an edge of $G$, then $\mid N(v) \cap N(w)\mid  = \mu$.
\end{itemize}

In \cite{P7}, the vertex set of a strongly regular graph is partitioned into four sets based on the neighborhoods of two distinct vertices (see Lemma 10 in \cite{P7}). We shall generalize this result in Lemma \ref{lemmasrg}, which is essential in proving Theorem \ref{Q12}.

\begin{lemma} \label{lemmasrg}
Let $G=(V(G),E(G))$ be a strongly regular graph with parameters $(n,k,\lambda,\mu)$. For $a,b \in V(G), a\neq b$, define
   \begin{center}
       $G_1=\{c\in V(G) \backslash \{a,b\}\mid ac\in E(G) \text{ and } bc \in E(G)\},$\\
       $G_2=\{c\in V(G) \backslash \{a,b\}\mid ac\notin E(G) \text{ and } bc \notin E(G)\},$\\
       $G_3=\{c\in V(G) \backslash \{a,b\}\mid ac\notin E(G) \text{ and } bc \in E(G)\},$\\
       $G_4=\{c\in V(G) \backslash \{a,b\}\mid ac\in E(G) \text{ and } bc \notin E(G)\}.$\\
   \end{center}
 Then $\lvert G_1 \rvert =\lambda$, $\lvert G_2 \rvert = n-2k+\lambda$, $\lvert G_3 \rvert = \lvert G_4 \rvert =k-\lambda-1$, if $ab\in E(G)$ and $\lvert G_1 \rvert =\mu$, $\lvert G_2 \rvert =n-2-2k+\mu$, $\lvert G_3 \rvert = \lvert G_4 \rvert =k-\mu$, if $ab\notin E(G)$. 
\end{lemma} 
\begin{proof} It is a general knowledge  that the complement of $G$, $\overline{G}$, is also a strongly regular graph with parameters $(n,\overline{k},\overline{\lambda},\overline{\mu})$, where $\overline{k}=n-k-1$, $\overline{\lambda}=n-2-2k+\mu$, and $\overline{\mu}=n-2k+\lambda$. 

Consider the following two cases, where $d(v)$ denotes the degree of the vertex $v$ in $G$.

\textit{Case 1}. $ab\in E(G)$, then
$\lvert G_1 \rvert =\lambda$. Since $ab\notin E(\overline{G})$, then $\lvert G_2\rvert = \lvert \{c\in V(G) \backslash \{a,b\}\mid ac\notin E(G) \text{ and } bc \notin E(G)\}\rvert = \lvert \{c\in V(\overline{G}) \backslash \{a,b\}\mid ac\in E(\overline{G}) \text{ and } bc \in E(\overline{G})\}\rvert 
=\overline{\mu}$. By considering the neighbours of $a$, we obtain $\lvert G_1\rvert + \lvert G_4\rvert + \lvert \{b\}\rvert =d(a)$, and hence $\lvert G_4\rvert = k-\lambda-1$.
By considering the neighbours of $b$, $\lvert G_1\rvert + \lvert G_3\rvert + \lvert \{a\}\rvert = d(b)$, and hence $\lvert G_3\rvert =k-\lambda-1$. Thus, $\lvert G_1\rvert =\lambda$, $\lvert G_2\rvert = \overline{\mu}$, $\lvert G_3\rvert = \lvert G_4\rvert =k-\lambda-1$.
    
\textit{Case 2}. $ab\not\in E(G)$, then $\lvert G_1\rvert =\mu$. Since $ab\in E(\overline{G})$, then $\lvert G_2\rvert =\lvert \{c\in V(G) \backslash \{a,b\}\mid ac\notin E(G) \text{ and } bc \notin E(G)\}\rvert = \lvert \{c\in V(\overline{G}) \backslash \{a,b\}\mid ac\in E(\overline{G}) \text{ and } bc \in E(\overline{G})\}\rvert =\overline{\lambda}$.
By considering the neighbours of $a$, we obtain $\lvert G_1\rvert + \lvert G_4\rvert =d(a)$, and hence $\lvert G_4\rvert =k-\mu$. And by considering the neighbours of $b$, $\lvert G_1\rvert + \lvert G_3\rvert =d(b)$, and hence $\lvert G_3\rvert =k-\mu$. And so, $\lvert G_1\rvert =\mu$, $\lvert G_2\rvert =\overline{\lambda}$, $\lvert G_3\rvert = \lvert G_4\rvert = k-\mu$. 
\end{proof}

In the following lemma, we shall show that the function $\theta$ defined in Theorem \ref{Q12} is always an integer. This function will be further investigated in Section \ref{OPD}. 
\begin{lemma}\label{newlemma}
If $G$ is a strongly regular graph with parameters $(n,k,\lambda, \mu)$ and $\theta=\max\{k/2,\lambda,\mu,(n-k-1)/2,$ $n-2-2k+\mu,n-2k+\lambda\}$, then $\theta$ is an integer.
\end{lemma}

\begin{proof}
Assume that there exists a strongly regular graph $G$ such that $\theta$ is not an integer. It is well known that for a strongly regular graph, the following equation holds:
\begin{equation}\label{musthold}
    (n-k-1)\mu=k(k-\lambda-1)
\end{equation}
Since $\overline{G}$ is a strongly regular graph with parameters $(n,\overline{k},\overline{\lambda},\overline{\mu})$, where $\overline{k}=n-k-1$, $\overline{\lambda}=n-2-2k+\mu$, and $\overline{\mu}=n-2k+\lambda$; then $\theta=\max\{k/2,\overline{k}/2\}$. Without loss of generality, assume that $k\geq \overline{k}$, and so, $\theta=k/2$ and $k$ should be odd.

\textit{Case 1.} $k>\overline{k}$: From Equation $(\ref{musthold})$,  $\mu>k-\lambda-1$ or $\lambda+\mu+1>k$. Since $\theta=k/2$ is not an integer, then $k/2>\lambda$ and $k/2>\mu$ and hence $k>\lambda+\mu$, a contradiction.

\textit{Case 2.} $k=\overline{k}$: Here $n=2k+1$ and $\overline{G}$ is a strongly regular graph with parameters $(n,k,\mu-1,\lambda+1)$. From Equation $(\ref{musthold})$, $k=\mu+\lambda+1$. Since $\theta=k/2$ is not an integer, then $k/2>n-2k+\lambda$ and $k/2>\mu$, and so $\mu-1>\lambda$ and $\lambda+1>\mu$. Thus, we have $\overline{\lambda}>\lambda$ and $\overline{\mu}>\mu$. Since $k=\overline{k}$, we may proceed with $\overline{G}$ instead which leads to $\lambda>\overline{\lambda}$ and $\mu>\overline{\mu}$, a contradiction. 
\end{proof}

The following two theorems will also be utilized in proving Theorem \ref{Q12}. The theorems were originally proved in \cite{P7} and \cite{P7new}. However, we rewrite the conditions of the theorems, adding those necessary but omitted to be included.




\begin{theorem}\label{batasatasZ}
Given positive integers $n$, $k\geq 2$, and $m\geq 2$. If $m\left(\left\lceil \frac{(n-1)k^2+k+2m-1}{m}\right\rceil-1\right)$ is divisible by $k$, then $$M_m(K_{2,n};k)\leq \left\lceil \frac{(n-1)k^2+k+2m-1}{m}\right\rceil.$$
\end{theorem} 

\begin{theorem}\label{newZ}
Given positive integers $c,k,n_1,\ldots,n_k$, $c,k,n_1\geq 2$, let $S=\sum_{i=1}^k n_i$. If $(c-1)\left\lceil \frac{ck(S-k)+(c-1)k}{(c-1)^2}\right\rceil$ is divisible by $k$, then
$$m_c(K_{2,n_1},K_{2,n_2},\ldots , K_{2,n_k})\leq \left\lceil \frac{ck(S-k)+(c-1)k}{(c-1)^2}\right\rceil.$$
\end{theorem}
Now, we are ready to prove Theorem \ref{Q12} by employing the method introduced in \cite{P7}. 

\textbf{Proof of Theorem \ref{Q12}}. 
Let $V(G)=\{1,2,\ldots,n\}$ and $H=[h_{i,j}]_{\zeta\times \zeta}$ be a symmetric $[\alpha]$-Hadamard matrix of order $\zeta\geq 2$. Consider the graph $K_{n \times \zeta}$, where its vertices are partitioned into $n$ classes $L_1,L_2,\ldots,L_n$, where $L_a=\{(a,1'),$  $(a,2'),\ldots,(a,\zeta')\}$, for $a\in \{1,2,\ldots,n\}$. Define an edge coloring on graph $K_{n\times \zeta}$ as follows:
$$\psi((a,i)(b,j))= 
\begin{cases}
 +h_{i,j} \text{ if } ab\in E(G),\\
 -h_{i,j} \text{ if } ab\notin E(G).
\end{cases}$$
Since $H$ is symmetric, then $\psi((a,i)(b,j))= \psi ((b,j)(a,i))$.
Let $w\in \{-1,1\}$, $v_1=(a,i)$, $v_2=(b,j)$, where $v_1\neq v_2$, and  $\delta(v_1,v_2,w)=\lvert \{v\in V(K_{n\times \zeta})\mid \psi(v_1v)=\psi(v_2v)=w\}\rvert$. 
We shall prove that $\delta (v_1,v_2,w)\leq \theta \zeta$.

 \textit{Case 1}. $a\neq b$ and $i=j$.  A vertex $(c,k)$ is simultaneously adjacent to $v_1=(a,i)$ and $v_2=(b,i)$ with color $w=1$ if and only if one of the following conditions hold:
\begin{enumerate}
        \item $c$ is simultaneously adjacent to $a$ and $b$ in $G$ and $h_{i,k}=1$,
        \item $c$ is simultaneously not adjacent to $a$ and $b$ in $G$ and  $h_{i,k}=-1$.
\end{enumerate}
Let $p$ be the number
of entries equal to $1$ in the $i^{th}$ row of $H$. Then, by Lemma \ref{lemmasrg}, $ \delta (v_1,v_2,1)=p\lvert G_1\rvert +(\zeta-p)\lvert G_2\rvert  \leq \max\{\lvert G_1\rvert ,\lvert G_2\rvert \}\zeta$. 
With similar reason, for $w=-1$, we obtain $\delta (v_1,v_2,w)\leq \max\{\lvert G_1\rvert ,\lvert G_2\rvert \}\zeta$.
    
\textit{Case 2}.  $a=b$ and  $i\neq j$. Since the vertex $a=b$ has degree $k$ in $G$ and has degree $\overline{k}$ in graph $\overline{G}$, then by Lemma \ref{lemmahadamard}, we get $\delta (v_1,v_2,1)=\lvert I_1\rvert k+ \lvert I_2\rvert \overline{k} \leq (\lvert I_1\rvert +\lvert I_2\rvert )\max\{k,\overline{k}\} \leq \frac{(\zeta+\alpha)}{2}\max\{k,\overline{k}\}.$
With a similar reason, for $w=-1$, we obtain
$\delta (v_1,v_2,-1) \leq \frac{(\zeta+\alpha)}{2}\max\{k,\overline{k}\}$.
    
\textit{Case 3.}  $a\neq b$ and  $i\neq j$. A vertex $(c,k)$ is simultaneously adjacent to $v_1=(a,i)$ and $v_2=(b,j)$ with color $w=1$ if and only if one of the following conditions holds:
\begin{center}
$ ac\in E(G),  bc \in E(G),  h_{i,k}=1 \text{ and }  h_{j,k}=1$,\\
$ ac\notin E(G),  bc \notin E(G),  h_{i,k}=-1 \text{ and }  h_{j,k}=-1$,\\
$ ac\in E(G),  bc \notin E(G),  h_{i,k}=1 \text{ and }  h_{j,k}=-1$, or\\
$ ac\notin E(G),  bc \in E(G),  h_{i,k}=-1 \text{ and }  h_{j,k}=1$.
\end{center}
Then $\delta (v_1,v_2,1)=\lvert G_1\rvert \lvert I_1\rvert +\lvert G_2\rvert \lvert I_2\rvert +\lvert G_3\rvert \lvert I_4\rvert +\lvert G_4\rvert \lvert I_3\rvert $. By Lemma \ref{lemmahadamard}, 
\begin{align*}
    \delta (v_1,v_2,1)
&\leq \max\{\lvert G_1\rvert,\lvert G_2\rvert ,\lvert G_3\rvert,\lvert G_4\rvert \}(\lvert I_1\rvert +\lvert I_2\rvert +\lvert I_3\rvert +\lvert I_4\rvert )\\
&=\max\{\lvert G_1\rvert,\lvert G_2\rvert,\lvert G_3\rvert,\lvert G_4\rvert \}\zeta.
\end{align*}
With a similar reason, for $w=-1$,
$\delta (v_1,v_2,-1)\leq \max\{\lvert G_1\rvert,\lvert G_2\rvert,\lvert G_3\rvert,\lvert G_4\rvert \}\zeta$.

From all the cases and by Lemma \ref{lemmasrg}, we conclude that 
\begin{align*}
    \delta (v_1,v_2,w)&\leq \max\{\frac{k}{2},\frac{\overline{k}}{2},\lvert G_1\rvert ,\lvert G_2\rvert,\lvert G_3\rvert,\lvert G_4\rvert \}(\zeta+\alpha)\\
    &=\max\{\frac{k}{2},\frac{\overline{k}}{2},\lambda,\overline{\lambda},\mu,\overline{\mu}\}(\zeta+\alpha)\\
    &=\theta (\zeta+\alpha).
\end{align*}
From Lemma \ref{newlemma}, $\theta(\zeta+\alpha)$ is an integer, and hence there is no monochromatic $K_{2,\theta (\zeta+\alpha)+1}$. Thus, we conclude that
$$M_\zeta(K_{2,\theta (\zeta+\alpha)+1};2)\geq n+1 \text{ and }  m_{n}(K_{2,\theta (\zeta+\alpha)+1};2)\geq \zeta+1.$$

Furthermore,
\begin{enumerate}
\item[\rm $i).$] From Theorem \ref{batasatasZ}, if $\zeta$ or $\left\lceil \frac{4\theta\alpha+1}{\zeta}  \right\rceil-1$ are even, then
$M_\zeta(K_{2,\theta (\zeta+\alpha)+1};2)\leq 4\theta +2 +\left\lceil \frac{4\theta\alpha+1}{\zeta}  \right\rceil$. 
\item[\rm $ii).$] From Theorem \ref{newZ}, if $n-1$ or $\left\lceil \frac{4n \theta(\zeta+\alpha)}{(n-1)^2}+\frac{2}{n-1}  \right\rceil$ are even, then we obtain $ m_{n}(K_{2,\theta (\zeta+\alpha)+1};2)\leq \left\lceil \frac{4n \theta(\zeta+\alpha)}{(n-1)^2}+\frac{2}{n-1}  \right\rceil$. \qed
\end{enumerate}



In the rest of the section, we employ Theorem \ref{Q12} for strongly regular graphs of particular parameters. For instance, we consider a strongly regular graph with parameters $(4n-3, 2n-2, n-2, n-1)$ in Example \ref{Z2}. In this case, $\theta=n-1$.

\begin{example}\label{Z2}
Suppose that there exist a  strongly regular graph with parameters $(4n-3, 2n-2, n-2, n-1)$ and a  symmetric $[\alpha]$-Hadamard matrix of order $\zeta$. If $\zeta$ or $\left\lceil \frac{4(n-1)\alpha+1}{\zeta}  \right\rceil-1$ are even, then
$$4n-2\leq M_\zeta(K_{2,(\zeta+\alpha)(n-1)+1};2) \leq  4n-2+\left\lceil \frac{4(n-1)\alpha+1}{\zeta}  \right\rceil.$$
Furthermore, if  $\zeta\geq 4(n-1)\alpha+1$, then
 $$4n-2\leq M_\zeta(K_{2,(\zeta+\alpha)(n-1)+1};2) \leq  4n-1.$$
Additionally, if $\alpha=0$ and $n\geq \frac{\zeta +6}{4}$, then $$m_{4n-3}(K_{2,\zeta (n-1)+1};2)=\zeta +1.$$
\end{example}

The $n\times n$ square Rook's graph, which is the line graph of a balanced complete bipartite graph $K_{n,n}$, is a strongly regular graph with parameters $(n^2,2n-2, n-2,2)$. If $n\geq 4$, then $\theta=(n-2)(n-1)$. 

\begin{example} \label{cor1}
Let $n\geq 4$ and suppose that there exists a symmetric $[\alpha]$-Hadamard matrix of order $\zeta\geq 2$. If $\zeta$ or $\left\lceil \frac{4(n-2)(n-1)\alpha+1}{\zeta}  \right\rceil-1$ are even, then $$n^2+1\leq M_\zeta(K_{2,(n-2)(n-1)(\zeta+\alpha)+1};2)\leq 4(n-2)(n-1) +2 +\left\lceil \frac{4(n-2)(n-1)\alpha+1}{\zeta}  \right\rceil.$$
Furthermore, if $\alpha=0$, then
$$n^2+1\leq M_\zeta(K_{2,(n-2)(n-1)\zeta+1};2)\leq 4n^2-12n+10.$$
\end{example}

The line graph of a complete graph $K_n$ is a strongly regular graph with parameters $(n(n-1)/2,2(n-2),n-2,4)$. For $n\geq 6$, $\theta=(n-3)(n-2)/2$. 

\begin{example} \label{cor2}
Let $n\geq 6$ and suppose that there exists a symmetric $[\alpha]$-Hadamard matrix of order $\zeta\geq 2$. If $\zeta$ or $\left\lceil \frac{2(n-3)(n-2)\alpha+1}{\zeta}  \right\rceil-1$ are even, then $\frac{n(n-1)}{2}+1\leq M_\zeta(K_{2,\frac{(n-3)(n-2)}{2}(\zeta+\alpha)+1};2)\leq 2(n-3)(n-2) +2 +\left\lceil \frac{2(n-3)(n-2)\alpha+1}{\zeta}  \right\rceil.$
Furthermore, if $\alpha=0$, then
$$\frac{n(n-1)}{2}+1\leq M_\zeta(K_{2,\frac{(n-3)(n-2)}{2}\zeta+1};2)\leq 2(n-3)(n-2)+3.$$
\end{example}

We observe that the gap between the lower and upper bounds of Example \ref{Z2} is $O(n)$; however, the gaps between the lower and upper bounds in Example \ref{cor1} and Example \ref{cor2} are $O(n^2)$. Thus, we propose the following.

\begin{openproblem}
Find all strongly regular graphs such that the gap between lower and upper bounds in Theorem \ref{Q12} is $O(n)$.
\end{openproblem}




Exoo \textit{et al.} \cite{P4} constructed a strongly regular graph with parameters $(4n-3,2n-2,n-2,n-1)$ when $4n-3=p^t$ for a prime $p$ and a positive integer $t$. Additionally, the existence of a symmetric Hadamard matrix of order $4r^4$ for all odd $r$ has been proved by Muzychuk and Xiang \cite{P6}. Combining these facts with the last result of Example \ref{Z2}, we obtain  the following two examples.

\begin{example}
Let $p$ be a prime number and $t$ be a positive integer with $p^t\equiv 1\pmod 4$. If there exists a symmetric Hadamard
matrix of order $4k$ where $p^t\geq 4k+3$, then
$$m_{p^t}(K_{2,k(p^t-1)+1};2)=4k+1.$$
\end{example}

\begin{example}
Let $p$ be a prime number and $t$ be a positive integer with $p^t\equiv 1\pmod 4$. Let $r$ be an odd number where $p^t\geq 4r^4+3$. Then
$$m_{p^t}(K_{2,r^4(p^t-1)+1};2)=4r^4+1.$$

Furthermore, for $r=1$ and $ p^t\geq 7$, we obtain the following. 
$$ m_{p^t}(K_{2,p^t},K_{2,p^t})=5.$$
\end{example}

Notice that $n=p^t$ is always odd for any prime $p>2$. For even $n$, the smallest known value is $m_{n}(K_{2,n},K_{2,n})=5$, 
proved as the bipartite Ramsey number $b(K_{2,2};K_{2,2})$ by
Beineke and Schwenk \cite{beineke}. 
Although we only know the exact value for $m_{2}(K_{2,2},K_{2,2})$, we suspect that the same value also applies to larger values of even $n$. Thus, we propose the following.
\begin{conjecture}\label{x}
$m_{n}(K_{2,n},K_{2,n})=5$ for $n\geq 2$.
\end{conjecture}

\section{Proof of Theorem \ref{A11}} \label{sec:exact}

We start this section by presenting Lemma \ref{A3}, which will be utilized to obtain the upper bound of Theorem \ref{A11}. 

\begin{lemma} \cite{P7} \label{A3}
Let $k,s,n_1,n_2,\ldots,n_k \in \mathbb{N}$ and $k\geq 2$. Suppose that $c$ is a positive integer such that
\begin{equation}\label{XX1}
    kcs \binom{\frac{(c-1)s}{k}}{2}> \sum_{i=1}^k (n_i-1)\binom{cs}{2},
\end{equation}
    then $M_s(K_{2,n_1}, K_{2,n_2},\ldots, K_{2,n_k})\leq c$.
\end{lemma}

\textbf{Proof of Theorem \ref{A11}}.
The lower bound is directly obtained from Theorem \ref{Q12}. 
By Lemma \ref{A3}, to prove that $M_\zeta(K_{2,(\zeta+\alpha)(n-1)+1};2) \leq  4n-2$, we need to show that $s=\zeta$, $k=2$, $n_1=n_2=(\zeta+\alpha)(n-1)+1$, and $c=4n-2$  satisfy the inequality (\ref{XX1}) or equivalently, we need to show that 
$$\zeta^2c^2-(2\zeta+4((\zeta+\alpha)(n-1)+1)-2)\zeta c+(\zeta^2+2\zeta+4((\zeta+\alpha)(n-1)+1)-4)>0$$

Now, let $p=2\zeta+4((\zeta+\alpha)(n-1)+1)-2$, $q=\zeta^2+2\zeta+4((\zeta+\alpha)(n-1)+1)-4$, and $t=4\alpha( n-1)+2$. Note that
     $$p=2\zeta+4((\zeta+\alpha)(n-1)+1)-2
      =4\zeta n-2\zeta +4\alpha n+2-4\alpha,$$ and
    $$q=p+\zeta ^2-2.$$ 
Consider the real function $f(x)=\zeta ^2x^2-p\zeta x+q$; we need to show that $f(4n-2)>0$. Note that the discriminant of $f$ is positive, as shown below.
$$p^2-4q=p^2-4(p+\zeta ^2-2)=
(p-2)^2+4(\zeta ^2-3)>0. $$

Therefore, $f(x)$ has two different roots. Since $f$ is a convex quadratic function and the highest root of $f$ is $(p+\sqrt{p^2-4q})/(2\zeta) $, then we need to show that
$4n-2>(p+\sqrt{p^2-4q})/(2\zeta) $, that is equivalent to
$$2\zeta (4n-2)-p>\sqrt{p^2-4q},$$
$$8\zeta n-4\zeta -(4\zeta n-2\zeta +4\alpha n+2-4\alpha)> \sqrt{p^2-4q},$$
$$4\zeta n-2\zeta -4\alpha n-2+4\alpha > \sqrt{p^2-4q},$$
$$(4\zeta n-2\zeta +4\alpha n+2-4\alpha)-8\alpha n-4+8\alpha > \sqrt{p^2-4q},$$
\begin{equation}\label{ine1}
  p-(8\alpha n+4-8\alpha) > \sqrt{p^2-4q}.
\end{equation}
The hypothesis
$\zeta >(\sqrt{2}+1)(2n-1)(4\alpha n-4\alpha+1)$
implies that
$$-4\alpha n+4\alpha-2>-\frac{\zeta }{(\sqrt{2}+1)(2n-1)}-1.$$
Consequently,
\begin{align*}
4\zeta n-2\zeta -4\alpha n-2+4\alpha &>4\zeta n-2\zeta -\frac{\zeta }{(\sqrt{2}+1)(2n-1)}-1\\
&>4\zeta n-2\zeta -\zeta -1=\zeta (4n-3)-1>0.
\end{align*}
This means that both sides in (\ref{ine1}) can be squared. Hence, inequality (\ref{ine1}) is equivalent to
$$p^2-2p(8\alpha n+4-8\alpha)+(8\alpha n+4-8\alpha)^2 > p^2-4q,$$
$$p(4\alpha n+2-4\alpha)+(4\alpha n+2-4\alpha)^2 < q,$$
$$pt+t^2 < p+\zeta ^2-2,$$
$$(4\zeta n-2\zeta +t)t+t^2 < 4\zeta n-2\zeta +t+\zeta ^2-2,$$
$$\zeta ^2+4\zeta n-2\zeta -4\zeta nt+2\zeta t-2t^2+t-2>0,$$
$$\zeta ^2-(4nt-4n+2-2t)\zeta -2t^2+t-2>0,$$
$$\zeta ^2-(4n(t-1)-2(t-1))\zeta -2t^2+t-2>0,$$
\begin{equation}\label{XX2} \zeta ^2-(4n-2)(t-1)\zeta -(2t^2-t+2)>0.  
\end{equation}
Now, we will prove that inequalities (\ref{XX2}) is true by considering the real function 
$$g(x)=x^2-(4n-2)(t-1)x-(2t^2-t+2).$$
The highest root of $g(x)$ is
$$x_1=(2n-1)(t-1)+\sqrt{(2n-1)^2(t-1)^2+(2t^2-t+2)},$$ 
Again, since $g$ is a convex quadratic function and the highest root of $g$ is $x_1$, then we need to show that $\zeta>x_1$. By hypothesis 
\begin{align*}
\zeta &> (\sqrt{2}+1)(2n-1)(4\alpha n-4\alpha+1)\\
&=(2n-1)(t-1)+\sqrt{(2n-1)^2(t-1)^2+(2n-1)^2(t-1)^2},
\end{align*}
it is enough to show that $(2n-1)^2(t-1)^2>(2t^2-t+2)$. Since $n,t\geq 2$, this is sufficient by proving that $9(t-1)^2>(2t^2-t+2)$ for every $t\geq 2$. But, this is easy due to the fact that 
$h(t)=9(t-1)^2-(2t^2-t+2)=7t^2-17t+7$ is positive
for every $t\geq 2$. 
We conclude that $(2n-1)^2(t-1)^2>(2t^2-t+2)$
for every $n,t\geq 2$. Hence, $\zeta>(\sqrt{2}+1)(2n-1)(t-1)>x_1$.

Since $g$ is a convex quadratic function with highest root $x_1$, where $\zeta>x_1$, then (\ref{XX2}) is true. This completes the proof. $\square$

The rest of the section will be dedicated to applying Theorems \ref{A11} in obtaining set multipartite Ramsey numbers for particular complete multipartite graphs. Combining Theorem \ref{A11} and Theorem \ref{H1}, we obtain the following examples.

\begin{example}\label{4n-22A}
Suppose that there exist a strongly regular graph with parameters $(4n-3, 2n-2, n-2, n-1)$ and a symmetric Hadamard matrix of order $\zeta$, where $\zeta-\alpha$ is even, $\alpha\leq \zeta /2$, and
$\zeta -\alpha> (\sqrt{2}+1)(2n-1)(4\alpha n-4\alpha+1)$. Then
$$M_{\zeta -\alpha}(K_{2,\zeta (n-1)+1};2) =  4n-2.$$
\end{example}

Utilizing Exoo's strongly regular graph with parameters $(4n-3,2n-2,n-2,n-1)$, where $4n-3=p^t$ for a prime $p$ and a positive integer $t$, and Muzychuk and Xiang's symmetric Hadamard matrix of order $4r^4$ for all odd $r$, Example \ref{4n-22A} leads to the following.

\begin{example}\label{F12}
Let $p$ be a prime number and $t$ be a positive integer with $p^t\equiv 1\pmod 4$. Let $r$ be odd number with $4r^4-\alpha$ is even, $\alpha\leq 2r^4$,
$4r^4-\alpha> (\sqrt{2}+1)((p^t+3)/2-1)(\alpha (p^t+3)-4\alpha+1)$.  Then
$$M_{4r^4-\alpha}(K_{2,r^4(p^t-1)+1};2)=p^t+1.$$    
\end{example}

\section{Open Problems}\label{OPD}

In this paper, we present some exact values of multipartite Ramsey numbers by utilizing strongly regular graphs of parameters $(4n-3, 2n-2, n-2, n-1)$ in Theorem \ref{Q12}. It is easy to check that a strongly regular graph $G$ and its complement $\overline{G}$ have the same parameters if and only if the parameters of $G$ is $(4n-3, 2n-2, n-2, n-1)$. In this case, the value of $\theta$ is minimized, that is, $\theta=n-1=k/2=(\lvert G\rvert -1)/4$. Since the upper bound in Theorem \ref{Q12} is a function of $\theta$, a better upper bound is achieved when the $\theta$s are small relative to $\lvert G\rvert $. 

In the case of a strongly regular graph $G$ with parameters $(4n-3, 2n-2, n-2, n-1)$, we have 
    \[\lim_{\lvert G\rvert \to \infty} \frac{\theta}{\lvert G\rvert }=\frac{1}{4}.\]
However, the limits for the other two strongly regular graphs considered in Section \ref{sec:Ramsey}, that is $L(K_{n,n})$ and $L(K_n)$, are greater than $\frac{1}{4}$. These observations lead us to the following interesting question. We suspect the answer to the question would be $\frac{1}{4}$, but we have no further evidence to form a conjecture.

\begin{openproblem}\label{CvZ}
Let $\mathcal{G}$ be the set of all strongly regular graphs with parameters $(n,k,\lambda,\mu)$ and $\theta=\max\{k/2,\lambda,\mu,(n-k-1)/2, n-2-2k+\mu,n-2k+\lambda\}$, find the value of
$$\lim_{n\to\infty} \inf \left\{  \frac{\theta}{n} \bigl\rvert\, G\in\mathcal{G} \right\}.$$
\end{openproblem}

Another interesting question is whether we could find other strongly regular graphs that can be utilized to obtain exact values of multipartite Ramsey numbers.

\begin{openproblem}
Is $(4n-3, 2n-2, n-2, n-1)$ the only parameters for a strongly regular graph such that the lower and upper bounds in Theorem \ref{Q12} coincide?
\end{openproblem}

\section*{Statements and Declarations}


\begin{itemize}
\item \textbf{Funding} I.W.P. Anuwiksa is supported by Ganesa Talent Assistantship - Research Group (GTA-RG) Scholarship. R. Simanjuntak and E.T. Baskoro are supported by Penelitian Multidisiplin Lintas KK, PPMI-FMIPA, 2021-2022.

\item \textbf{Competing interests} The authors have no relevant financial or non-financial interests to disclose.
\item \textbf{Availability of data}
Our manuscript has no associated data.
\item \textbf{Author contributions} All authors contributed to the study conception and design. Material preparation, data collection and analysis were performed by P. Anuwiksa, R. Simanjuntak, and E.T. Baskoro. The first draft of the manuscript was written by P. Anuwiksa and all authors commented on previous versions of the manuscript. All authors read and approved the final manuscript.
\end{itemize}








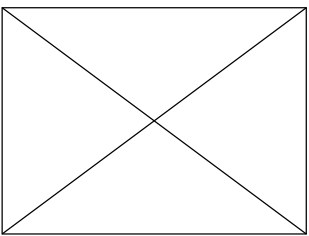
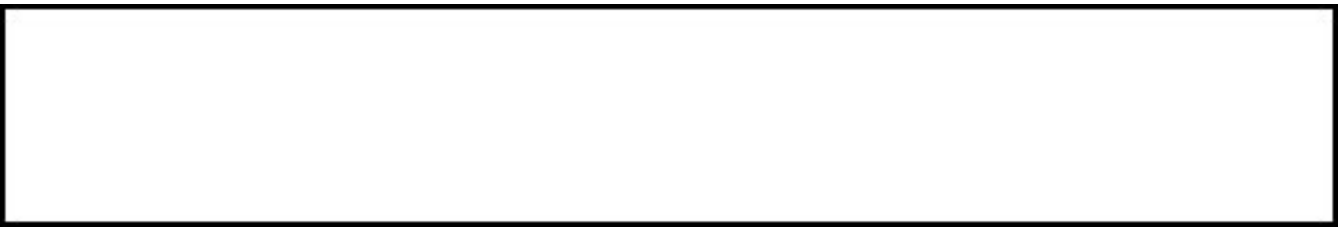
\end{document}